\documentclass{amsart}
\usepackage{amsmath}
\usepackage{amsfonts}
\usepackage{amsthm}
\usepackage{amssymb}
\usepackage{enumerate}
\usepackage[all]{xy}
\usepackage{cite}
\usepackage{mathrsfs}
\usepackage{color}
\usepackage{xcolor}
\usepackage{hyperref}
\usepackage{fancyhdr}
\hypersetup{
	colorlinks=true,
	anchorcolor=blue,
	linkcolor=blue,
	filecolor=blue,
	urlcolor=blue,	
	citecolor=blue,
	bookmarksopen=true,
	pdfborder={0 0 0}
}
\numberwithin{equation}{section}
\theoremstyle{plain}
\newtheorem{thm}{Theorem}[section]

\newtheorem{proposition}[thm]{Proposition}

\newtheorem{cor}[thm]{Corollary}

\newtheorem{lemma}[thm]{Lemma}
\newtheoremstyle{noparens}%
 {}{}%
 {\itshape}{}%
 {\bfseries}{.}%
 { }%
 {\thmname{#1}\thmnumber{ #2}\mdseries\thmnote{#3}}
\theoremstyle{noparens}

\theoremstyle{definition}
\newtheorem{defn}[thm]{Definition}

\theoremstyle{remark}
\newtheorem{rmk}[thm]{Remark}

\makeatletter

\newcommand{\Rmnum}[1]{\expandafter\@slowromancap\romannumeral #1@}

\newcommand{\B}{\mathbb B}
\newcommand{\C}{\mathbb{C}}
\newcommand{\R}{\mathbb R}
\newcommand{\Aut}{\operatorname{Aut}}\newcommand{\PSH}{\operatorname{PSH}}
\newcommand{\Vol}{\operatorname{Vol}}
\newcommand{\cL}{\mathcal L}
\newcommand{\Imag}{\operatorname{Im}}

\makeatother
\begin{document}
\title{Finite volume quotients of bounded domains with piecewise smooth boundary}
\author{Xingsi Pu\textsuperscript{1} $\&$ Lang Wang\textsuperscript{2} }
\address{$1.$ Mathematical Science Research Center, Chongqing University of Technology, Chongqing, 400054, China}
\address{$2.$ School of Mathematical Sciences, Guizhou Normal University, Guiyang, 550025, P.R. China.}
\email{puxs@cqut.edu.cn,\:wanglang2020@amss.ac.cn}


\begin{abstract}
We prove that a bounded pseudoconvex domain with piecewise
$C^2$-smooth boundary which has a finite volume quotient for the
complete K\"ahler--Einstein metric is biholomorphic to a product of
complex unit balls. More generally, the same conclusion holds for a
bounded domain possessing a boundary point satisfying Kodama's
conditions (C.1)--(C.5), provided the quotient has finite volume for
a complete $\Aut(D)$-invariant K\"ahler metric. For piecewise Levi-flat boundary the product is a
polydisk and in the smooth pseudoconvex case it is a single ball.
\end{abstract}

\maketitle
\section{Introduction}

A central problem in several complex variables is to understand how
boundary geometry and automorphisms determine a bounded domain.
The Wong--Rosay theorem characterizes the ball by an automorphism
orbit accumulating at a strongly pseudoconvex boundary point
\cite{W77,R79}. At a nonsmooth boundary point several hypersurface
faces may meet, and products of balls are natural models.
Pinchuk developed the relevant piecewise smooth boundary theory
\cite{pinchuk2,pinchuk} and generalized this characterization to products of balls under the assumption of homogeneity.
Subsequently, Kodama obtained a local realization
theorem using the conditions (C.1)--(C.5) recalled below \cite{K86,KodamaII}.

For convex domains, Frankel proved that a compact quotient forces
the domain to be bounded symmetric \cite{fr}. For generic convex domains with piecewise Levi flat boundary, Kim \cite{pf} characterized those possessing a non-compact automorphism group, and Fu-Wong \cite{fw} later extended this to two-dimensional simply connected generic domains with piecewise smooth Levi flat boundary. 
Zimmer established
the ball characterization for smoothly bounded pseudoconvex domains with
finite-volume quotients \cite{finite}.

In our previous work \cite{PW26}, we obtain a bounded generic convex domain with piecewise $C^2$-smooth boundary but not smooth in $\mathbb{C}^2$ which has finite volume is biholomorphic to the bidisk.
Recently, Kuai--Liu--Wu \cite{KLW26} prove that a bounded convex domain that admits a finite volume quotient is biholomorphic to a bounded symmetric domain, so is a homogeneous domain. By the result of Pinchuk \cite{pinchuk2},
it must be the products of balls under the condition of piecewise $C^2$-smooth boundary. 
In this paper, we drop the convexity assumption and get the following main result.

\begin{thm}\label{ab}
Let $D\subset\C^n$, $n\ge2$, be a bounded pseudoconvex
domain with piecewise $C^2$-smooth boundary. Let
$\Gamma\le\Aut(D)$ be discrete and act freely. If $\Gamma\backslash D$ has finite volume for
the complete canonical K\"ahler--Einstein metric, then
\[
D\simeq\B^{n_1}\times\cdots\times\B^{n_k},
\quad \sum_{j=1}^k n_j=n.
\]
\end{thm}

To prove the above theorem, we need to establish the following more general result.

\begin{thm}\label{fv}
Let $D\subset\C^n$, $n\ge2$, be bounded and possess a  boundary point
$p$ satisfying (C.1)--(C.5) with $k$ local defining functions.
Suppose that $D$ carries a complete $\Aut(D)$-invariant K\"ahler
metric $g$ and that a discrete group $\Gamma\le\Aut(D)$ acts freely
 with $\Vol_g(\Gamma\backslash D)<\infty$.
Then
\[
D\simeq\prod_{j=1}^k\B^{n_j},\quad\sum_j n_j=n.
\]
\end{thm}
\begin{rmk}
	Note that we have obtained that the Teichmüller space cannot be biholomorphic to a domain with a good corner point on its boundary \cite[Theorem 1.10]{PW26}. Since the Teichmüller space also has a finite-volume quotient with respect to the complete canonical  K\"ahler--Einstein metric, this provides a new proof.
\end{rmk}

Under the Levi flat condition on the smooth boundary, we can obtain the polydisk case without using Theorem \ref{ab}.
\begin{thm}\label{thm:flat}
Let $D\subset\C^n$, $n\ge2$, be a bounded domain with
piecewise $C^2$-smooth Levi flat boundary, and let $\Gamma\le\Aut(D)$ be
discrete and act freely. If $\Gamma\backslash D$ has finite volume for
the complete canonical K\"ahler--Einstein metric, then $D\simeq\mathbb D^n$.
\end{thm}

Moreover, Theorem \ref{fv} can imply the following corollary, which was originally stated  in \cite[Theorem 1.3]{finite}.
\begin{cor}\label{cor:smooth}
Let $D\subset\C^n$, $n\ge2$, be a bounded pseudoconvex domain
with $C^2$-smooth boundary, and let $\Gamma\le\Aut(D)$ be
discrete and act freely. If
$\Gamma\backslash D$ has finite Bergman, K\"ahler--Einstein,
or Kobayashi--Eisenman volume, then $D\simeq\B^n$.
\end{cor}

The paper is organized as follows. In Section \ref{sec2} we will recall some definitions and preliminary results. In Section \ref{sect3} we give the proofs of main theorems.

\section{Preliminaries}\label{sec2}
\subsection{Notations}

\begin{enumerate}
    \item $|\cdot|$ denotes the standard Euclidean norm in $\mathbb{C}^n$.

    \item $\mathbb{D}$ is the unit disk in $\mathbb{C}$, $\mathbb{B}^n$ and $\mathbb{D}^n$ are the unit ball and unit polydisk in $\mathbb{C}^n$ for $n\geq2$ respectively.

    \item For each $n\geq1$, $B_r(z):=\{w\in\mathbb{C}^n:|z-w|<r\}$ and $D_r(z):=\{w\in\mathbb{C}^n:|z_i-w_i|<r,i=1,\cdots,n\ \operatorname{with}\ z=(z_1,\cdots,z_n)\}$. 
    
    \item For a real $C^2$ function $\rho$, the Levi form is
    \[
    \cL_\rho(p;\xi,\eta)=\sum_{\alpha,\beta=1}^n \rho_{z_\alpha\bar z_\beta}(p)\xi_\alpha\overline{\eta_\beta}.
    \]
    The kernel of a Hermitian form $L$ is the space of $\xi$ such that $L(\xi,\eta)=0$ for every $\eta$. If $L\ge0$, this equals the zero set of $L(\xi,\xi)$.

    \item We write $k_D$ for the Kobayashi distance and $\kappa_D$ for its infinitesimal metric, normalized by $\kappa_{\B^n}(0;v)=|v|$. Their distance-decreasing property and biholomorphic invariance are used throughout. For a bounded pseudoconvex domain $D$, let $g_B$ and $g_{KE}$ denote the Bergman and complete K\"ahler--Einstein metrics respectively. We normalize $\operatorname{Ric}(g_{KE})=-g_{KE}$. 
\end{enumerate}

The existence and uniqueness of $g_{KE}$ are supplied by
Cheng--Yau and Mok--Yau \cite{CY80,MY83}. Uniqueness implies
$f^*g_{KE}=g_{KE}$ for $f\in\Aut(D)$, since pullback preserves
completeness and the normalized Einstein equation. We denote the
Riemannian volume forms of $g_B$ and $g_{KE}$ by $\mu_B$ and
$\mu_{KE}$. With $d\lambda$ the standard Euclidean volume form,
our Kobayashi--Eisenman volume convention is
\[
\mu_K(z)=\inf\left\{\frac{1}{|\det f'(0)|^2}:
 f\in\mathcal O(\B^n,D),\ f(0)=z\right\}d\lambda(z).
\]
A zero determinant contributes $+\infty$ to the infimum.

\subsection{Piecewise smooth boundary and good corner point}

Following \cite{pinchuk}, a bounded domain $\Omega\subset\mathbb{C}^n$ has \emph{piecewise $C^r$-smooth boundary} if there exists a neighborhood $U$ of $\overline\Omega$ and $C^r$-smooth functions $\rho_1,\cdots,\rho_k:U\rightarrow\mathbb{R}$ such that
\[
\Omega=\{z\in U:\rho_1(z)<0,\cdots,\rho_k(z)<0\}
\]
and for any possible distinct indices $i_1,\cdots,i_l$, the following
\[
d\rho_{i_1}\wedge d\rho_{i_2}\wedge\cdots\wedge d\rho_{i_l}\neq0 
\]
holds at common boundary points of the corresponding active faces.  

Let $\rho:\mathbb{C}^n\rightarrow\mathbb{R}$ be a $C^2$-smooth function such that 
\[
\operatorname{grad}\rho(p)\neq0\ \operatorname{with}\ \rho(p)=0.
\]
The hypersurface $H=\{z\in\C^n:\rho(z)=0\}$ is \emph{Levi flat} at $p\in H$ if
\[
\sum_{i,j=1}^n\frac{\partial^2\rho}{\partial z_i\partial\overline{z}_j}(p)w_i\overline{w}_j=0
\]
for any $w=(w_1,\cdots,w_n)$ whenever
\[
\sum_{i=1}^n\frac{\partial\rho}{\partial z_i}(p)w_i=0.
\]
Moreover, $H$ is called Levi flat, if it is Levi flat at every point. For a bounded domain with piecewise $C^2$-smooth boundary  $\Omega=\{z\in U:\rho_1(z)<0,\cdots,\rho_k(z)<0\}$, its boundary $\partial\Omega$ is called \emph{piecewise Levi flat} if the hypersurface $\{z\in U:\rho_i(z)=0\}$ is Levi flat for each $i=1,\cdots,k$.

Latter, Kodama \cite{K86} uses the following local five hypotheses. 
For a bounded domain \(D\) in \(\C^n\) \((n\geq2)\), we say a boundary point \(p\in\partial D\) is \emph{good corner} if there exist an open neighborhood \(U\) of \(p\) and real-valued \(C^2\)-functions \(\rho_1,\dots,\rho_k\) \((1\le k\le n)\) defined on \(U\) such that
\begin{align}
	&\rho_1(p)=\cdots=\rho_k(p)=0; \tag{C.1}\\
	&D\cap U=\{z\in U:\rho_i(z)<0,\quad i=1,\dots,k\}; \tag{C.2}\\
	&\bar\partial\rho_1\wedge\cdots\wedge\bar\partial\rho_k(z)\ne 0
	\qquad\text{for all }z\in U; \tag{C.3}\\
	&\sum_{\alpha,\beta=1}^{n}
	\frac{\partial^2\rho_i}{\partial z_\alpha\partial\bar z_\beta}(p)
	\xi_\alpha\bar\xi_\beta\ge 0,\quad
	\xi\in \bigcap_{j=1}^k\ker\partial\rho_j(p),\quad i=1,\dots,k, \tag{C.4}
\end{align}
where
\[
\ker\partial\rho_j(p)=\left\{\xi=(\xi_\alpha)\in\C^n:
\sum_{\alpha=1}^{n}\frac{\partial\rho_j}{\partial z_\alpha}(p)\xi_\alpha=0\right\};
\]
for some constant \(A\ge 0\), the function
\[
\rho=\sum_{i=1}^{k}\rho_i+A\sum_{i=1}^{k}\rho_i^2 \tag{C.5}
\]
is strictly plurisubharmonic on \(U\). 

Pinchuk's existence result \cite[Lemma 1.2--1.3]{pinchuk} states that every bounded pseudoconvex
domain \(D\) with piecewise \(C^2\)-smooth boundary admits a
boundary point \(p\in\partial D\) satisfying (C.1)--(C.5).
The local face functions are selected from the active faces of the original defining system. The construction supplies complex linear independence at the selected point, which persists on a neighborhood by continuity. Real independence of the differentials in the piecewise smooth definition alone does not imply (C.3).

\subsection{Siegel domains}
According to Piateckii--Sapiro \cite{PC66} (see also \cite{K86}), 
we first recall the definition of an \(\Omega\)-hermitian form.
For a given open convex cone \(\Omega\subset\R^n\) containing no full straight line, a mapping \(H\colon\C^m\times\C^m\to\C^n\) is called an \emph{\(\Omega\)-hermitian form} if
\begin{enumerate}
	\item \(H\) is complex linear in the first variable;
	\item \(\overline{H(u,v)}=H(v,u)\) for all \(u,v\in\C^m\);
	\item \(H(u,u)\in\overline{\Omega}\) for all \(u\in\C^m\), where \(\overline{\Omega}\) denotes the topological closure of \(\Omega\) in \(\R^n\);
	\item \(H(u,u)=0\) if and only if \(u=0\).
\end{enumerate}
Then the  definition of a Siegel domain is as follows.
\begin{defn}
	For a given open convex cone \(\Omega\subset\R^n\) containing no full straight line and an \(\Omega\)-hermitian form \(H\colon\C^m\times\C^m\to\C^n\), the domain
	\[
	\mathcal{D}(\Omega,H)=\{(z,w)\in\C^n\times\C^m:
	\Imag z-H(w,w)\in\Omega\}
	\]
	in \(\C^n\times\C^m\) is called a \emph{Siegel domain of the second kind}  associated with \(\Omega\) and \(H\). When \(m=0\) it reduces to the domain
	\[
	\mathcal{D}(\Omega)=\{z\in\C^n:\Imag z\in\Omega\}.
	\]
	Such a domain is called a \emph{Siegel domain of the first kind}.
\end{defn}

In this paper we mainly study  a special case of a Siegel domain of the first or the second kind.
For the open convex cone
\[
\R_+^k=\{(y_1,\dots,y_k)\in\R^k:y_i>0,\quad i=1,\dots,k\}
\]
in \(\R^k\) \((1\le k\le n)\) and an \(\R_+^k\)-hermitian form \(H\colon\C^{n-k}\times\C^{n-k}\to\C^k\), let \(\mathcal{D}(\R_+^k,H)\) denote the Siegel domain in \(\C^k\times\C^{n-k}\) associated to \(\R_+^k\) and \(H\).
Writing $H=(H_1,\ldots,H_k)$,
each $H_j$ is semipositive and $H(w,w)=0$ implies $w=0$.
For semipositive forms, this is equivalent to $\sum_jH_j>0$.
Thus, the model contains no affine complex lines. 
The case $n=k$ has no
$w$-variables and is simply the product of upper half-planes.
Then we have the following.

\begin{lemma}[\cite{PC66}]\label{PCL}
	There exists a biholomorphic mapping from \(\mathcal{D}(\R_+^k,H)\) onto a subdomain \(D\) of the direct product of the open unit balls
	\[
	\mathbb{B}^{n_1}\times\cdots\times\mathbb{B}^{n_k},
\quad n_j\ge1,\quad \sum_{j=1}^k n_j=n.
	\]
\end{lemma}

We remark that every homogeneous bounded domain in \(\C^n\) is pseudoconvex and that \(\B^{n_1}\times\cdots\times\B^{n_k}\) is biholomorphically equivalent to the so-called Siegel domain
\[
\mathcal{E}(n_1,\dots,n_k)
=\{(z,\omega)\in\C^k\times\C^{n-k}:
\Imag z_i-|\omega_i|^2>0,\ i=1,\dots,k\}
\]
in \(\C^k\times\C^{n-k}=\C^k\times\C^{n_1-1}\times\cdots\times\C^{n_k-1}\), where \(|\cdot|\) denotes the Euclidean norm on \(\C^{n_i-1}\) for each $i$.
For a domain \(D\subset\C^n\), we denote by \(\Aut(D)\) the group of all biholomorphic automorphisms of \(D\).  
We now recall the following realization theorem due to Kodama.

\begin{thm}[\cite{K86}, Theorem I]\label{thm:kodama}
	Let \(D\subset\C^n(n\geq2)\) be a bounded domain  with a boundary point \(p\in\partial D\) satisfying the conditions from (C.1) through (C.5). Assume that \((*)\) holds: 
	there exist a compact set \(K\subset D\), a sequence \(\{k_n\}\subset K\) and a sequence \(\{f_n\}\subset\Aut(D)\) such that	
	$$\lim_{n\to\infty}f_n(k_n)=p.$$
	Then \(D\) is biholomorphically equivalent to a Siegel domain \(\mathcal{D}(\R_+^k,H)\) in \(\C^k\times\C^{n-k}\). Conversely, every Siegel domain \(\mathcal{D}(\R_+^k,H)\) in \(\C^k\times\C^{n-k}\) is biholomorphically equivalent to a bounded domain \(D\subset\C^n\) satisfying all the conditions (C.1)--(C.5) and \((*)\).
\end{thm}

\subsection{PSH peak point}

A point $p\in\partial D$ is a \emph{plurisubharmonic(PSH) peak point} if there
exists a negative plurisubharmonic function $u$ on $D$ such that
$u(z)\to0$ as $z\to p$ and $\sup_{D\setminus U}u<0$ for every
neighborhood $U$ of $p$. 

\begin{lemma}\label{lem:peak}
Let $D\subset\C^n$ be a bounded domain, and let $p\in\partial D$ satisfy (C.1)--(C.5). Then there exists $u\in\PSH(D)\cap C(D)$ with $-1\le u<0$ such that
	\[
	\lim_{D\ni z\to p}u(z)=0,\quad \sup_{D\setminus B(p,\varepsilon)}u<0\quad(\varepsilon>0).
	\]
\end{lemma}
\begin{proof}
	After translating $p$ to $0$, shrink the neighborhood so that $\rho_i\le0$ on the local closure of $D$ and, if $A>0$, $|\rho_i|<1/(2A)$ for every $i$. Then
	\[
	\rho=\sum_i\rho_i(1+A\rho_i)\le0
	\]
	on that closure, and $\rho<0$ on $D$ in this neighborhood.
	Since $\rho$ is strictly PSH, we have 
	$$\cL_\rho(0;\xi,\xi)\ge c|\xi|^2$$ 
	for some $c>0$. Let
	\[
	F(z)=\sum_\alpha \rho_{z_\alpha}(0)z_\alpha+\frac12\sum_{\alpha,\beta}\rho_{z_\alpha z_\beta}(0)z_\alpha z_\beta.
	\]
	The $C^2$ Taylor formula gives
	\[
	\rho(z)=2\Re F(z)+\cL_\rho(0;z,z)+o(|z|^2)\ge2\Re F(z)+\frac c2|z|^2
	\]
	on a sufficiently small neighborhood. Since $\rho\le0$ on the local closure of $D$, with $c_0=c/4$ we obtain
	\begin{equation}\label{eq:h}
		h(z):=\Re F(z)\le-c_0|z|^2
	\end{equation}
	for $z\in\overline D$ in that neighborhood. 
	
	Choose $r>0$ such that $\overline{B(0,r)}$ is contained in this neighborhood, then choose $0<\delta<r/2$ and $0<a<c_0(r-\delta)^2$. Define
	\[
	u(z)=\begin{cases}
		\max\{h(z)/a,-1\},&z\in D\cap B(0,r),\\
		-1,&z\in D\setminus B(0,r).
	\end{cases}
	\]
	On the collar $r-\delta\le|z|<r$ inside $D$, inequality \eqref{eq:h} gives $h/a<-1$. 
	Hence the two formulas coincide on an open collar near the artificial boundary. Since $h$ is pluriharmonic, the local maximum construction gives $u\in\PSH(D)\cap C(D)$. It is strictly negative on $D$ and tends to $0$ at $p$. Moreover, for $0<\varepsilon<r$,
	\[
	\sup_{D\setminus B(0,\varepsilon)}u\le-\min\{1,c_0\varepsilon^2/a\}<0,
	\]
	and when $\varepsilon\ge r$ the function equals $-1$ on the relevant set. 
\end{proof}

\subsection{Finite volume Liouville and the orbit envelope}

We need the following bounded plurisubharmonic Liouville lemma, which is stated in \cite[Lemma~2.1]{KLW26}. For completeness, we justify the argument for bounded PSH functions using local Sobolev regularity and also prove a nonnegative \(L^2\) version.
\begin{lemma}\label{lem:liouville}
	Let $(M,g)$ be a connected complete K\"ahler manifold of finite Riemannian volume. Then every bounded real-valued plurisubharmonic function on $M$ is constant.
	More generally, without the finite volume assumption, every bounded nonnegative plurisubharmonic function $w$ satisfying $\int_M w^2\,d\mu_g<\infty$ is constant.
\end{lemma}
\begin{proof}
	For the finite volume assertion, we add a constant and rescale such that $0\le w\le1$. The resulting function still belongs to $L^2(M,g)$. For the more general assertion, we only rescale, which preserves both nonnegativity and square integrability. Thus, in either case, we may assume $0\le w\le1$ and $w\in L^2(M,g)$. We first show that $w\in W^{1,2}_{\rm loc}(M)$. Indeed, in a holomorphic coordinate ball, convolve $w$ on smaller balls with nonnegative smooth approximate identities. The functions $w_\delta$ are smooth, nonnegative, bounded by $1$, and PSH, hence Euclidean subharmonic. For a smooth compactly supported cutoff $\eta$, integration by parts gives
	\[
	\int \eta^2|\nabla w_\delta|^2\,d\lambda
	\le2\int \eta w_\delta|\nabla\eta|\,|\nabla w_\delta|\,d\lambda
	\le\frac12\int\eta^2|\nabla w_\delta|^2\,d\lambda+2\int w_\delta^2|\nabla\eta|^2\,d\lambda.
	\]
	Thus the local $W^{1,2}$ norms of $w_\delta$ are uniformly bounded. Since convolution converges in $L^2_{\rm loc}$, weak compactness proves the claim. On compact coordinate subsets, the smooth metric and the Euclidean metric are comparable.

	Since the positive current $i\partial\bar\partial w$ has nonnegative trace with respect to the K\"ahler form, we have $\Delta_gw\ge0$ weakly, where $\Delta=\operatorname{div}\nabla$. Completeness provides compactly supported Lipschitz cutoffs $\chi_R$ equal to $1$ on $B_g(o,R)$, vanishing outside $B_g(o,2R)$, with $0\le\chi_R\le1$ and $|\nabla\chi_R|\le2/R$ almost everywhere. The nonnegative function $\chi_R^2w$ is an admissible Sobolev test function, after approximation by nonnegative smooth compactly supported tests. The passage to these tests is justified because $w\in W^{1,2}_{\rm loc}$ makes the functional $\psi\mapsto-\int\langle\nabla w,\nabla\psi\rangle$ continuous in the $W^{1,2}$ norm on every relatively compact coordinate neighborhood; a partition of unity gives the same statement on a compact support. Hence
	\[
	0\le-\int_M\langle\nabla w,\nabla(\chi_R^2w)\rangle\,d\mu_g,
	\]
	which yields the Caccioppoli inequality
	\[
	\int_M\chi_R^2|\nabla w|^2\,d\mu_g
	\le4\int_Mw^2|\nabla\chi_R|^2\,d\mu_g
	\le\frac{16}{R^2}\int_M w^2\,d\mu_g.
	\]
	In the finite-volume case the last integral is bounded by $\Vol_g(M)$. 
	On every fixed compact set, the gradient therefore vanishes in $L^2$ as $R\to\infty$. Thus, $w$ is constant almost everywhere. If this value is $c$, the submean inequality in coordinate balls gives $w\le c$, while upper semicontinuity and density of the full-measure set $\{w=c\}$ give $w\ge c$. Hence $w$ is everywhere constant.
\end{proof}

\begin{lemma}\label{lem:envelope}
	Let $\{u_\alpha\}$ be a nonempty family of PSH functions on a domain, all taking values in $[-1,0]$. If $v=\sup_\alpha u_\alpha$, then
	\[
	V(z)=v^*(z):=\limsup_{y\to z}v(y)
	\]
	is PSH and takes values in $[-1,0]$.
\end{lemma}
\begin{proof}
	The upper semicontinuous regularization $V(z)=v^*(z)$ is bounded and upper semicontinuous.
  Fix a complex disc $z+\zeta\xi$, $|\zeta|\le r$, compactly contained in the domain. Choose $z_\nu\to z$ and indices $\alpha_\nu$ such that $u_{\alpha_\nu}(z_\nu)\to V(z)$. For $\nu$ sufficiently large, the submean inequality on the parallel discs centered at $z_\nu$ gives
	\[
	u_{\alpha_\nu}(z_\nu)\le\frac1{2\pi}\int_0^{2\pi}V(z_\nu+re^{it}\xi)\,dt.
	\]
	Applying reverse Fatou Lemma to these uniformly bounded-above functions and using the upper semicontinuity of $V$,
	we obtain
	\[
	V(z)\le\frac1{2\pi}\int_0^{2\pi}V(z+re^{it}\xi)\,dt.
	\]
	Thus $V$ is subharmonic on every complex line and hence PSH.
\end{proof}

The finite volume Liouville lemma turns a PSH peak point into an orbit accumulation point as follows.
\begin{proposition}\label{prop:orbit}
	Let $D\subset\C^n$ be a bounded domain. Suppose taht a discrete group $\Gamma\le\Aut(D)$ acts freely and preserves a complete K\"ahler metric whose quotient has finite volume. If $p\in\partial D$ admits a function $u$ as in Lemma~\ref{lem:peak}, then for every $q\in D$ there are $\gamma_\nu\in\Gamma$ with $\gamma_\nu q\to p$.
\end{proposition}
\begin{proof}
	Set
	\[
	v(z)=\sup_{\gamma\in\Gamma}u(\gamma z),\qquad V=v^*.
	\]
	Then Lemma~\ref{lem:envelope} gives $V\in\PSH(D)$ and $-1\le V\le0$.
	For every \(\delta\in\Gamma\), right multiplication by \(\delta\) permutes \(\Gamma\), so \(v\circ\delta=v\).
	Since $\delta$ is a homeomorphism,
\[
V(\delta z)=\limsup_{y\to\delta z}v(y)=\limsup_{x\to z}v(\delta x)=V(z).
\]
   Thus $V$ is invariant and descends to a bounded PSH function on the quotient \(\Gamma\backslash D\). By Lemma~\ref{lem:liouville}, it is constant. Since $V\ge v\ge u$ and $\sup_Du=0$, while $V\le0$, the constant is $0$.

	Fix \(q\in D\). From \(v^*(q)=0\), choose \(q_\nu\to q\) such that \(v(q_\nu)>-1/\nu\), and then choose \(\gamma_\nu\in\Gamma\) with \(u(\gamma_\nu q_\nu)>-2/\nu\). The uniform negative bound away from \(p\) in Lemma~\ref{lem:peak} forces \(\gamma_\nu q_\nu\to p\).
	
	To replace \(q_\nu\) by \(q\), choose \(R_0>0\) with \(D\subset B(0,R_0)\). For the Kobayashi distance, normalized so that \(\kappa_{\B^n}(0;v)=|v|\), monotonicity and the ball metric give
	\begin{equation}\label{eq:ballcontrol}
		|x-y|\le R_0 k_{B(0,R_0)}(x,y)\le R_0k_D(x,y).
	\end{equation}
	The first inequality follows because the infinitesimal ball metric is at least \(|v|/R_0\). Also \(k_D(q_\nu,q)\to0\), for example by comparing with a fixed Euclidean ball centered at \(q\) and contained in \(D\). Automorphism invariance and \eqref{eq:ballcontrol} now imply
	\[
	|\gamma_\nu q_\nu-\gamma_\nu q|\le R_0k_D(q_\nu,q)\to0.
	\]
	Hence \(\gamma_\nu q\to p\). 
\end{proof}

\section{The proofs of main theorems}\label{sect3}

To prove Theorem \ref{fv},  we first prove the following Lemma.

\begin{lemma}\label{lem:split}
Let
\[
S=\mathcal D(\R_+^k,H)=\{(z,w)\in\C^k\times\C^m:
\Imag z_j>H_j(w,w),\ 1\le j\le k\},
\]
where $H_j\ge0$ and $\sum_jH_j$ is positive definite on $\C^m$.
When $m=0$, the latter condition is vacuous. Suppose $S$ admits a
complete $\Aut(S)$-invariant K\"ahler metric $g$, and that there exists a discrete group \(\Lambda\le\Aut(S)\) acting freely such that
$\Vol_g(\Lambda\backslash S)<\infty$. Then
\[
S\simeq\prod_{j=1}^k\B^{n_j},\quad\sum_jn_j=k+m.
\]
\end{lemma}
\begin{proof}
	If $m=0$, then $S$ is the product of $k$ upper half-planes,
	hence is biholomorphic to $\mathbb D^k$. So we may assume
	$m\ge1$.
	
	By Lemma~\ref{PCL}, $S$ admits a bounded realization.
	Transporting $g$ and $\Lambda$ via the realizing biholomorphism
	preserves completeness, invariance, discreteness, freeness,
	and quotient volume. It also identifies the automorphism
	groups and their identity components.
	
	Set $G_0=\Aut_0(S)$. The lattice lemma
	\cite[Lemma~3.1]{KLW26}, which applies to a complete invariant
	K\"ahler metric, yields that $\Lambda_0=\Lambda\cap G_0$
	is a lattice in $G_0$. Consequently, $G_0$ is unimodular.
	
	We apply the Lie algebra argument in the proof of
	\cite[Proposition~6.2]{Vey}.
	For $\mathfrak g=\operatorname{Lie}(G_0)$, the standard
	grading associated with the Euler field
	\[
	E=\sum_{j=1}^k z_j\frac{\partial}{\partial z_j}
	+\frac12\sum_{\alpha=1}^m
	w_\alpha\frac{\partial}{\partial w_\alpha}
	\]
	is
	\[
	\mathfrak g
	=\mathfrak g_{-1}\oplus\mathfrak g_{-1/2}
	\oplus\mathfrak g_0\oplus\mathfrak g_{1/2}
	\oplus\mathfrak g_1.
	\]
	Its real dimensions satisfy
	\[
	\dim\mathfrak g_{-1}=k,\quad
	\dim\mathfrak g_{-1/2}=2m,\quad
	\dim\mathfrak g_{1/2}\le2m,\quad
	\dim\mathfrak g_1\le k.
	\]
	Unimodularity implies
	\[
	0=\operatorname{tr}(\operatorname{ad}E)
	=-k-m+\frac12\dim\mathfrak g_{1/2}
	+\dim\mathfrak g_1,
	\]
	so both upper bounds are attained.
	Vey's argument then shows that $\mathfrak g$
	spans the real tangent space at each $(iy,0)$,
	$y\in\mathbb R_+^k$. Real and Heisenberg translations in
	$G_0$ move every point of $S$ to such a point.
	Hence every $G_0$-orbit is open. Since $S$ is connected,
	$G_0$ acts transitively on $S$.
	
	For any $x\in S$, the compact set $K=\{x\}$ therefore
	satisfies $\Aut(S)K=S$. The proposition in the
	proof of \cite[Corollary~1]{K86} yields a
	biholomorphism
	\[
	S\simeq
	\prod_{j=1}^k
	\left\{(\zeta_j,u_j)\in\mathbb C\times\mathbb C^{n_j-1}:
	\Imag\zeta_j>\|u_j\|^2\right\},
	\quad n_j\ge1,\quad \sum_jn_j=k+m.
	\]
	
	For each factor, the Cayley map is
	\[
	(\zeta,u)\longmapsto
	\left(\frac{\zeta-i}{\zeta+i},\frac{2u}{\zeta+i}\right).
	\]
	The identity
	\[
	|\zeta+i|^2-|\zeta-i|^2-4\|u\|^2
	=4(\Imag\zeta-\|u\|^2)
	\]
	shows that it maps into $\mathbb B^{n_j}$.
	Conversely, for $(a,b)\in\mathbb B^{n_j}$, set
	\[
	\zeta=i\frac{1+a}{1-a},\quad
	u=\frac{ib}{1-a}.
	\]
	These expressions are well-defined and satisfy
	\[
	\Imag\zeta-\|u\|^2
	=\frac{1-|a|^2-\|b\|^2}{|1-a|^2}>0.
	\]
	They define the holomorphic inverse, proving surjectivity.
	This includes $n_j=1$, where the $u$-coordinate is absent.
\end{proof}

\begin{thm}
Under the hypotheses of Theorem~\ref{fv},
$D\simeq\prod_{j=1}^k\B^{n_j}$ for positive integers $n_j$
with $\sum_jn_j=n$.
\end{thm}
\begin{proof}
Let $p$ and $\rho_1,\ldots,\rho_k$ satisfy (C.1)--(C.5).
By Lemmas~\ref{lem:peak}, $p$ admits a bounded plurisubharmonic
peak function. Then using Proposition~\ref{prop:orbit}, for any fixed $q\in D$, there exist
$\gamma_\nu\in\Gamma$ with $\gamma_\nu(q)\to p$.

Theorem~\ref{thm:kodama} applies with the compact set $K=\{q\}$.
It provides a biholomorphism $F:D\to S=\mathcal D(\R_+^k,H)$,
where $H$ is nondegenerate and $m=n-k$.
Define $g_S=(F^{-1})^*g$ and
$\Lambda=F\Gamma F^{-1}$. The fact that completeness, group is discrete and acting freely, and quotient volume are preserved by $F$.
Moreover, for $a\in\Aut(S)$, the automorphism $F^{-1}aF$ of $D$
preserves $g$, so $a^*g_S=g_S$. Thus all hypotheses of
Lemma~\ref{lem:split} hold. Its conclusion proves the theorem.
\end{proof}

\begin{proof}[Proof of Theorem~\ref{ab}]
The boundary point existence result in Section~\ref{sec2} gives a point \(p\in\partial D\) satisfying (C.1)--(C.5). By existence and uniqueness, the canonical K\"{a}hler--Einstein metric is complete and \(\Aut(D)\)-invariant. The discrete free action and the finite-volume assumption are precisely those in Theorem~\ref{fv}, which therefore yields the asserted product of balls.
\end{proof}

We now prove Theorem \ref{thm:flat} without using Theorem \ref{ab}.

\begin{thm}
Under the hypotheses of Theorem~\ref{thm:flat}, $D\simeq\mathbb D^n$.
\end{thm}
\begin{proof}
Note the domain is pseudoconvex, since it is  the intersection of domains with Levi flat boundary. 
This ensures the existence of the complete canonical
K\"ahler--Einstein metric  by \cite{MY83}.
The good corner point existence result consequently
supplies $p$ and $\rho_1,\ldots,\rho_k$ satisfying (C.1)--(C.5).
By Proposition \ref{prop:orbit} and Theorem \ref{thm:kodama}, \(D\)
is biholormorphic to a Siegel domain.

Set \(V=\bigcap_j\ker\partial\rho_j(p)\) and \(\rho=\sum_j\rho_j+A\sum_j\rho_j^2\) as in (C.5). For every \(v\in V\), Levi-flatness gives \(\cL_{\rho_j}(p;v,v)=0\) for all \(j\). Since \(\rho_j(p)=0\),
\[
\cL_\rho(p;v,v)=\sum_j\cL_{\rho_j}(p;v,v)
+2A\sum_j|\partial\rho_j(p)v|^2=0.
\]
Strict plurisubharmonicity of \(\rho\) forces \(v=0\). Condition (C.3) gives \(\dim_{\mathbb C}V=n-k\), so \(k=n\). Hence the Siegel domain is
\[
W=\{w\in\mathbb C^n:\operatorname{Re}w_i>0,\ i=1,\dots,n\},
\]
which is biholomorphic to \(\mathbb D^n\).
\end{proof}

\begin{proof}[Proof of Corollary~\ref{cor:smooth}]
	
Fix \(a\in D\) and choose \(p\in\overline D\) maximizing \(|p-a|\). Put \(R=|p-a|\). Then \(p\in\partial D\), \(D\subset B(a,R)\), and the sphere containing \(D\) is tangent to \(\partial D\) at \(p\). As in \cite[Lemma~4.1]{GS20}, comparison with this sphere shows that \(p\) is a strongly convex boundary point, hence strongly pseudoconvex. Moreover,
\[
\overline D\cap T_p\partial D
\subset \overline{B(a,R)}\cap T_p\partial B(a,R)=\{p\},
\]
so \(p\) is globally strongly convex. Here \(T_p\partial D\) denotes the real affine tangent hyperplane to \(\partial D\) at \(p\).

Let \(r\) be a local defining function at \(p\). Its Levi form is positive on \(\ker\partial r(p)\), so \(r+Ar^2\) is strictly PSH near \(p\) for sufficiently large \(A\ge0\). Thus (C.1)--(C.5) hold at \(p\) with \(k=1\).

In the Kähler--Einstein volume case, Theorem~\ref{fv} now applies with \(k=1\) and gives \(D\simeq\B^n\). In the Bergman volume case, the Bergman metric is a complete \(\Aut(D)\)-invariant Kähler metric by Ohsawa's theorem \cite{O81}, since \(D\) is bounded and pseudoconvex with \(C^2\)-smooth boundary. Theorem~\ref{fv}, applied to the same point with \(k=1\), again gives \(D\simeq\B^n\).

It remains to consider the Kobayashi--Eisenman volume case.	
	
Let \(s_D\) denote the squeezing function, namely the supremum of \(r\) for which there exists a holomorphic embedding \(f:D\to\B^n\) satisfying \(f(z)=0\) and \(r\B^n\subset f(D)\). The boundary estimate at globally strongly convex points in \cite{sque} yields a neighborhood \(U\) of \(p\) and \(c\in(0,1)\) such that
\begin{equation}\label{eq:squeezing-local}
	s_D(z)\ge c\quad(z\in D\cap U).
\end{equation}
We use the following pointwise consequence of the squeezing comparisons for metrics and intrinsic measures in \cite{sque}: on the set
\[
A_c=\{z\in D:s_D(z)\ge c\},
\]
there exists \(C=C(n,c)\ge1\) such that, for any two of the volume forms \(\mu_B,\mu_{KE},\mu_K\),
\begin{equation}\label{eq:squeezing-volumes}
	C^{-1}\mu_\beta\le\mu_\alpha\le C\mu_\beta
	\quad\text{on }A_c.
\end{equation}

The cited comparisons include a comparison between the latter two measures; fixed normalization constants are absorbed into \(C\). More explicitly, let \(\widetilde\mu_\kappa\) denote the measure whose density is the reciprocal of the Euclidean volume of the infinitesimal Kobayashi unit indicatrix, with a fixed dimensional normalization. The intrinsic-measure estimate compares \(\mu_K\) with \(\widetilde\mu_\kappa\) by powers of \(s_D\), whereas the metric estimate compares the norms of \(g_B\) and \(g_{KE}\) with \(\kappa_D\). Comparing their unit indicatrices yields the corresponding volume inequalities. 

On \(A_c\), all these factors are bounded in terms of \(n\) and \(c\) alone, proving \eqref{eq:squeezing-volumes} even before a global squeezing bound is known. Equivalently, for any \(z\in A_c\), one can choose an embedding with
\[
f(z)=0,\quad (c/2)\B^n\subset f(D)\subset\B^n.
\]
Using \(c/2\) avoids the assumption that the supremum defining \(s_D(z)\) is attained. Ball monotonicity for the intrinsic measures and the squeezing-controlled metric estimates yield constants independent of \(z\); the Jacobian of \(f\) cancels in ratios of the invariant volume densities.

Let \(\mu\) be the volume form for which \(\Vol_\mu(M)<\infty\), where \(M=\Gamma\backslash D\). We first produce the orbit accumulation without assuming that \(\Vol_{KE}(M)\) is finite. Since \(p\) is strongly pseudoconvex, Lemma~\ref{lem:peak} gives a PSH peak function \(u\) with \(-1\le u<0\), chosen so that \(u=-1\) outside \(U\). This follows by choosing the ball used in the construction of Lemma~\ref{lem:peak} to have closure contained in \(U\); its annular gluing region does not affect the peak at \(p\). Set
\[
v(z)=\sup_{\gamma\in\Gamma}u(\gamma z),\quad
V=v^*,\quad W=1+V.
\]

As in Proposition~\ref{prop:orbit}, the function \(W\) is \(\Gamma\)-invariant, plurisubharmonic, and satisfies \(0\le W\le1\). We claim that
\begin{equation}\label{eq:envelope-support}
	\{W>0\}\subset A_c.
\end{equation}
Indeed, if \(W(z)>0\), then \(V(z)>-1\), so by the definition of \(v^*\) there exist \(z_\nu\to z\) with \(v(z_\nu)>-1\). Choose \(\gamma_\nu\in\Gamma\) with \(u(\gamma_\nu z_\nu)>-1\). Then \(\gamma_\nu z_\nu\in U\), so \eqref{eq:squeezing-local} and the biholomorphic invariance of \(s_D\) give \(s_D(z_\nu)\ge c\). The continuity of \(s_D\) \cite[Theorem~3.1]{DGZ12} then proves \eqref{eq:envelope-support}.

Both \(A_c\) and the volume comparisons descend to \(M\). Consequently,
\[
\int_M W^2\,d\mu_{KE}
\le C\int_M W^2\,d\mu
\le C\Vol_\mu(M)<\infty.
\]
Here the first inequality is valid because \(W=0\) outside the image of \(A_c\), and no comparison on its complement is needed. We use the same symbols for invariant functions on \(D\) and their descended counterparts on \(M\). The complete Kähler--Einstein metric descends to a complete metric on \(M\), and \(M\) is connected. Thus the nonnegative \(L^2\) version of Lemma~\ref{lem:liouville} applies. Explicitly, its cutoff argument uses the finiteness of the integral of \(W^2\) instead of the finiteness of the total Kähler--Einstein volume:
\[
\int_M\chi_R^2|\nabla W|_{KE}^2\,d\mu_{KE}
\le\frac{16}{R^2}\int_M W^2\,d\mu_{KE}\longrightarrow0.
\]

Thus $W$ is constant. Since $V\ge u$, $\sup_Du=0$, and $V\le0$, this constant is $1$. The preceding integrability estimate consequently gives
\[
\Vol_{KE}(M)=\int_M W^2\,d\mu_{KE}<\infty.
\]
The canonical K\"ahler--Einstein metric is complete and invariant, and the point $p$ satisfies (C.1)--(C.5) with $k=1$. Theorem~\ref{fv} therefore gives $D\simeq\B^n$.
\end{proof}

\vspace*{3mm}
\noindent {\bf Funding}.
Xingsi Pu was supported by the National Natural Science Foundation of China (Grant No. 12501091), the Science and Technology Research Program of Chongqing Municipal Education Commission (Grant No. KJQN202501110), Chongqing Natural Science Foundation Innovation and Development Joint Fund (CSTB2025NSCQ-LZX0059), Scientific Research Foundation of the Chongqing University of Technology (No. 2024ZDZ027).

\bibliography{ref}
\bibliographystyle{plain}
\end{document}